\documentclass[preprint]{imsart}

\RequirePackage{amsthm,amsmath,amsfonts,amssymb}
\RequirePackage[numbers]{natbib}
\RequirePackage[colorlinks,citecolor=blue,urlcolor=blue]{hyperref}
\RequirePackage{graphicx}
\RequirePackage{xcolor}
\RequirePackage{commath}

\startlocaldefs

\newtheorem{theorem}{Theorem}[section]
\newtheorem{proposition}[theorem]{Proposition}

\newtheorem{lemma}[theorem]{Lemma}
\theoremstyle{remark}
\newtheorem{definition}[theorem]{Definition}

\ifx\example\undefined
\newtheorem{example}[theorem]{Example}

\fi

\endlocaldefs

\begin{document}
	
	\begin{frontmatter}
		\title{Properties of statistical depth with respect to 
			compact convex random sets. The Tukey depth.}
		\runtitle{Tukey depth with respect to compact convex random sets}
		
		\begin{aug}
			\author[A]{\fnms{LUIS} \snm{GONZ\'ALEZ-DE LA FUENTE}\ead[label=e1]{gdelafuentel@unican.es}},
			\author[A]{\fnms{ALICIA} \snm{NIETO-REYES}\ead[label=e2]{alicia.nieto@unican.es}}
			\and
			\author[B]{\fnms{PEDRO} \snm{TER\'AN}\ead[label=e3]{teranpedro@uniovi.es}}
			\address[A]{Departamento de Matem\'aticas, Estad\'istica y Computaci\'on,
				Universidad de Cantabria (Spain),
			}
			
			\address[B]{Departamento de Estad\'istica e Investigaci\'on Operativa y Did\'actica de las Matem\'aticas,
				Universidad de Oviedo (Spain),
			}
		\end{aug}
		
		\begin{abstract}
			 We study a statistical data depth with respect to 
			compact convex random sets which is consistent with the multivariate Tukey depth and the Tukey depth for fuzzy sets.
			In doing so, we provide a series of properties 
			for  statistical data depth with respect to 
			compact convex random sets. These properties are an adaptation of properties that constitute the axiomatic notions of  multivariate, functional and fuzzy depth functions and other well-known properties of depth.	\end{abstract}

		
	\end{frontmatter}
	
	\section{Introduction}
	In some real cases  statistical data is
	in the form of sets, for instance, in the form of compact convex sets. Examples can be found in data sets related with health, such as the range of blood pressure over a day \citep{GilBlood}, or related with sport measures,  such as the range of weight and height of a soccer team \citep{soccer}.
	This type of statistical data is studied by the theory of random sets,  which, from the statistical point of view,  models observed phenomena that are sets rather than  points in $\mathbb{R}^p$, as in multivariate statistics, or functions as in functional data analysis.
	Thus, a random set is a generalization of a random variable: it is a set-valued random variable. A random set can also be understood as a simplification of a fuzzy random variable, as the $\alpha$-levels of a fuzzy set are nested compact sets.
	The literature about random sets contains well stablished theoretical results \citep{molchanov}, some of them being generalizations to random sets of classical statistical results, for instance, the strong law of large numbers \citep{arstein}.
	Statistical methods are also part of the development of  the area of compact convex random sets, such as proposing linear regression methods  \citep{gonzalezcolubi} or the median of a random interval \citep{sinovamedian}.
	Recent literature also includes theoretical results, as on the intersection of random sets \citep{intersect}, and applications, as on underwater sonar images \citep{Pen}.

	Statistical depth functions have become a very useful tool in nonparametric statistics.
	Nowadays  depth functions are applied in different fields of statistics such as clustering and classification  \citep{clustering} or real data analysis
	\citep{aplicacion1, aplicacion2}.
	Giving a distribution $\mathbb{P}$
	on a space, a depth function, $D(\cdot;\mathbb{P})$, orders the elements in the space with respect to  $\mathbb{P}$. Roughly speaking, statistical depth functions measure how close  an element is to a data cloud, in the sense that if we move the element to the \emph{center} of the cloud, its depth increases and if we move it out of the \emph{center}, its depth decreases.
	Assuming it is unique, this center is the center of symmetry, if the distribution is symmetric for a particular notion of symmetry.
	For multivariate spaces, there are notions of symmetry widely used in the literature: central, angular \citep{LiuSimplicial} and halfspace symmetry \citep{ZuoSerfling}. Notions of symmetry specific for functional \citep{NietoBatteyJMVA} and fuzzy spaces \citep{primerarticulo} are, however, quite recent.

	Formally, an axiomatic definition of  depth function for the multivariate case was proposed by \citet{ZuoSerfling}. According to it, a depth function, $D(\cdot;\mathbb{P})$, satisfies the following properties. To introduce them, let $X$ be a random variable with distribution $\mathbb{P}$ on $\mathbb{R}^{n},$  $\mathcal{M}_{n\times n}(\mathbb{R})$ the space of $n\times n$ matrices with entries in $\mathbb{R}$ and  $\|\cdot\|$ the Euclidean norm. Abusing the notation, we indistinctly write  $D(\cdot;X)$ and $D(\cdot;\mathbb{P}).$
	\begin{itemize}
		\item[M1.] Affine invariance.  A depth function does not depend on the coordinate system, that is, for any non-singular $M\in\mathcal{M}_{n\times n}(\mathbb{R})$   and $b\in\mathbb{R}^{n}$,  $D(x;X) = D(Mx + b; MX + b)$.
		
		\item[M2.] Maximality at center. If the distribution $\mathbb{P}$ has a uniquely defined center of symmetry, for certain notion of symmetry,  $D(\cdot;X)$ is maximized at it.
		
		\item[M3.] Monotonicity relative to the deepest point. Let $x_{0}\in\mathbb{R}^{n}$ be a point of maximal depth. Then, for any $x\in\mathbb{R}^{n}$, $D((1-\lambda)x_{0} + \lambda x;X)\geq D(x;X)$ for all $\lambda\in [0,1]$.
		
		\item[M4.] Vanishing at infinity. The limit of $D(x;X)$ goes to $0$ as the limit of $\|x\|$ goes to infinity.
	\end{itemize}
	Formal axiomatic definitions of depth function have been later provided in the  functional \citep{NietoBattey} and fuzzy settings \citep{primerarticulo, CaptFN}.
	
	The first instance of depth function was proposed  prior to the, axiomatic, definitions. It was an instance provided in 1975 by \citet{tukey} for multivariate data and is still nowadays the most well-known one. It  is also known as halfspace depth, as it  computes the infimum of the probabilities of closed halfspaces which contain the point at which the depth function is evaluated. That is,
	\begin{equation}\label{HD}
		HD(x;\mathbb{P}) := \inf\{\mathbb{P}(H) : \text{ H is a closed halfspace and } x\in H\}.
	\end{equation}
	\citet{ZuoSerfling}  proved that $HD$ satisfies M1-M4 and, therefore, it is a statistical depth function.
	We emphasize the satisfaction of the axioms because it is customary in the statistical depth community to not consider the axioms as cut-off, regarding a function as a depth function even when all the axioms are not satisfied in their entirety.
	
	Since Tukey coined the term in 1975,
	many other instances of depth functions have been proposed and their use in nonparametric statistics has grown considerably.
	Some commonly used
	are the simplicial depth, proposed by \citet{LiuSimplicial},  the spatial depth, proposed by \citet{SerflingSpatial} and the random Tukey depth,  proposed by \citet{CuestaNieto}, which, being based on random projections, is a computationally effective approximation of the Tukey depth.  The spatial and random Tukey depth functions can be applied in both multivariate and functional spaces \citep{SpatIndInf, CuestaNietoMieres}. However, the random Tukey depth does not satisfy the axiomatic definition of a functional depth \citep{NietoBattey}, which only the metric depth \citep{NietoBatteyJMVA} has yet been proved to satisfy. It is worth noticing that the spatial and  random Tukey  depth functions were introduced before the functional axiomatic definition.
	Furthermore, while the Tukey depth has not yet being defined in functional spaces, it has being generalized to the fuzzy setting and proved to satisfy the  axiomatic definitions in that setting  \citep{primerarticulo, CaptFT}.

	The aim of this paper is to
	propose some desirable properties of depth with respect to compact convex random sets, which can be considered as an axiomatic definition for this setting. Some of these properties are an adaptation for compact convex sets of those proposed in \citet{primerarticulo} for fuzzy data.
	The properties are also largely inspired by the multivariate definition  \citep{ZuoSerfling}, and, in addition, by the functional one \citep{NietoBattey}, because  the set of compact convex sets can be considered as a metric space by using the Hausdorff distance, for instance.
	In order to test the viability of those properties, with a generalization of halfspaces suitable for the space of compact convex sets, we present an adaptation of Tukey depth and show that almost all of them are satisfied. These definitions  of  halfspace and Tukey depth can be regarded to stem naturally from their corresponding multivariate definitions and, in addition, are  a particular case
	of their fuzzy analogs  \citep{primerarticulo}. Furthermore, we show that   the definition of Tukey depth with respect compact convex random sets coincides with that derived recently in \citet{Cascos}, which does not make an explicit use  of halfspaces in its definition. Moreover, we also show that  the Tukey depth with respect to compact convex random sets can  be rewritten in terms of the multivariate halfspace depth of the support function of compact convex sets.

	The paper is organized as follows.
	Background about compact convex random sets is contained in Section \ref{preliminaries}. The definition of the Tukey depth with respect compact convex random sets   is in Section \ref{defTukey}, together with its relations and equivalences to other definitions. Section \ref{properties} presents and studies the properties of depth with respect to compact convex random sets and their satisfaction by the Tukey depth with respect compact convex random sets. The paper concludes with some final remarks in Section \ref{remarks}.
	
	\section{Preliminaries on compact convex random sets.}\label{preliminaries}
	
	Let us denote by $\mathcal{K}_{c}(\mathbb{R}^{p})$ the set of non-empty
	compact convex sets of $\mathbb{R}^{p}$. 
	In the case $p = 1$, the elements of $\mathcal{K}_{c}(\mathbb{R})$ are intervals of the form $[a,b]$ with $a\leq b$.
	For any $K\in\mathcal{K}_{c}(\mathbb{R}^{p}),$ its support function $s_{K} : \mathbb{S}^{p-1}\rightarrow\mathbb{R}$ is defined by
	\begin{equation}\nonumber
		s_{K}(u) := \sup_{k\in K}\langle k,u\rangle,
	\end{equation}
	where $\langle\cdot ,\cdot\rangle$ denotes the usual dot product,   $\mathbb{S}^{p-1} := \{x\in\mathbb{R}^{p} : \|x\| = 1\}$ is the unit sphere, and $\|\cdot\|$ is the Euclidean norm.
	
	Let $(\Omega,\mathcal{A},\mathbb{P})$ be a probability space. A map $$\Gamma : \Omega\rightarrow\mathcal{K}_{c}(\mathbb{R}^{p})$$ is called a \textit{compact convex random set} if $$\{\omega\in\Omega : \Gamma(\omega)\cap K\neq\emptyset\}\in\mathcal{A}$$ for all $K\in\mathcal{K}_{c}(\mathbb{R}^{p})$  \cite{randomsetsdefinition}.
	\citet{randomsetstheorem} proved the Fundamental Measurability Theorem, which is useful to prove that
	$s_{\Gamma}(u)$ is a real random variable for all $u\in\mathbb{S}^{p-1}$.
	As in the  Euclidean space, in $\mathcal{K}_{c}(\mathbb{R}^{p})$ there exists a predominant distance, the Hausdorff metric. The Hausdorff distance between $K\in\mathcal{K}_{c}(\mathbb{R}^{p})$ and $L\in\mathcal{K}_{c}(\mathbb{R}^{p})$ is
	\begin{equation}\nonumber
		d_{\mathcal{H}}(K,L) := \max\{\sup_{k\in K}\inf_{l\in L} \|k-l\|, \sup_{l\in L}\inf_{k\in K}\|k-l\|\},
	\end{equation}
	which  can be expressed in terms of their support function (e.g., \cite{Bonnensen}) as
	\begin{equation}\label{ecuacionHausdorff}
		d_{\mathcal{H}}(K,L) = \sup_{u\in\mathbb{S}^{p-1}} |s_{K}(u) - s_{L}(u)|.
	\end{equation}
	Borel measurability with respect to $d_{\mathcal H}$ is equivalent to the above-mentioned definition of compact convex random sets.
	
	Some properties of support functions of elements of $\mathcal{K}_{c}(\mathbb{R}^{p})$ can be deduced from properties of the supremum function. For instance, let $K,L\in\mathcal{K}_{c}(\mathbb{R}^{p})$, taking into account that $$K + L = \{k+l : k\in K, l\in L\}\in\mathcal{K}_{c}(\mathbb{R}^{p}),$$ we have that the support function of $K + L$ can be expressed as the sum of the support functions of $K$ and $L$, that is, $$s_{K+L}(u) = s_{K}(u) + s_{L}(u)$$ for all $u\in\mathbb{S}^{p-1}$. It is also possible to define the product of $K$ by a scalar $\gamma\in\mathbb{R}^{+}$, as $$\gamma\cdot K = \{\gamma k : k\in K\}.$$ Then, it is clear that $$s_{\gamma\cdot K}(u) = \gamma \cdot s_{K}(u)$$ for all $u\in\mathbb{S}^{p-1}$.
	
	\section{Halfspaces and halfspace depth  in $\mathcal{K}_{c}(\mathbb{R}^{p})$}\label{defTukey}
	
	As observable from \eqref{HD}, the Tukey depth of a multivariate point $x$ is the infimum of the probability of halfspaces which contain $x$. But $\mathcal{K}_{c}(\mathbb{R}^{p})$ is not a linear space.   In this section, we define generalized halfspaces (called simply halfspaces in the sequel) for $\mathcal{K}_{c}(\mathbb{R}^{p})$ in a natural way from the multivariate case.
	
	Let $S$ be a halfspace of $\mathbb{R}^{n}$. Then, there exists $v\in\mathbb{R}^{n}$ and $b\in\mathbb{R}$ such that $$S = \{y\in\mathbb{R}^{n} : v^{T}y\leq b \}.$$ Taking $u = (1/\|v\|)v\in\mathbb{S}^{p-1}$ and $c = b/\|v\|$, it is clear that $$S = \{y\in\mathbb{R}^{n} : u^{T} y\leq c \}.$$ Thus, halfspaces of $\mathbb{R}^{n}$ can be viewed as subsets $S_{u,c}\subseteq\mathbb{R}^{n}$ such that $$S_{u,c} = \{y\in\mathbb{R}^{n} : u^{T} y\leq c \}$$ with $u\in\mathbb{S}^{p-1}$ and $c\in\mathbb{R}$.
	This generalizes naturally to $\mathcal{K}_{c}(\mathbb{R}^{p})$ by using the support function of a set.
	Thus, we define halfspaces $S_{u,t}^{-}, S_{u,t}^{+}\subseteq\mathcal{K}_{c}(\mathbb{R}^{p})$ as
	\begin{equation}\label{i}
		S_{u,t}^{-} := \{K\in\mathcal{K}_{c}(\mathbb{R}^{p}) : s_{K}(u)\leq t \},
	\end{equation}
	\begin{equation}\label{d}
		S_{u,t}^{+} := \{K\in\mathcal{K}_{c}(\mathbb{R}^{p}) : s_{K}(u)\geq t \},
	\end{equation}
	for all $u\in\mathbb{S}^{p-1}$ and $t\in\mathbb{R}$. We explicitly consider both halfspaces because
	$$s_{K}(-u) = -\inf_{k\in K}\langle u,k\rangle\neq -s_K(u)$$
	with
	\begin{equation}
		\begin{aligned}\nonumber
			S_{u,t}^{+}&\subseteq S_{-u,-t}^{-}, \\ \nonumber
			S_{u,t}^{-}&\subseteq S_{-u,-t}^{+},
		\end{aligned}
	\end{equation}
	for all $u\in\mathbb{S}^{p-1}$ and $t\in\mathbb{R}$.
	
	Making use of  both directions of the inequality that defines the halfspaces,
	the Tukey depth with respect to a compact convex random set can be defined.
	Let $\Gamma$ be a compact convex random set. \textit{The Tukey depth of} $K\in\mathcal{K}_{c}(\mathbb{R}^{p})$\textit{ with respect to} $\Gamma$ is defined by the function $$D_{CT}(\cdot;\Gamma) : \mathcal{K}_{c}(\mathbb{R}^{p})\rightarrow [0,1]$$ given by
	\begin{equation}\label{definicionTukeyCompacto}
		D_{CT}(K;\Gamma) := \min\{\inf_{\underset{K\in S_{u,t}^{-}}{u\in\mathbb{S}^{p-1},t\in\mathbb{R} :}}\mathbb{P}(\Gamma\in S_{u,t}^{-}), \inf_{\underset{K\in S_{u,t}^{+}}{u\in\mathbb{S}^{p-1}, t\in\mathbb{R} :}}\mathbb{P}(\Gamma\in S_{u,t}^{+}) \}.
	\end{equation}
	We indistinctively refer to it as Tukey depth for compact convex random sets or Tukey depth with respect to compact convex random sets.
	It is worth noticing that \eqref{definicionTukeyCompacto} is a particularization
	for compact convex sets of the Tukey depth for fuzzy sets proposed in \cite{primerarticulo}, as well as \eqref{i} and \eqref{d} are of the fuzzy halfspaces proposed there.
	
	In what follows, we operate on  \eqref{definicionTukeyCompacto} to show it coincides with the definition of half-space  depth with respect to compact convex random sets
	provided in  \citet{Cascos}, which does not explicitly use halfspaces.
	By \eqref{i},
	$K\in S_{u,t}^{-}$ means that $(u,t)$ is a pair such that $s_{K}(u)\leq t$. Thus
	$$S_{u,s_{K}(u)}^{-}\subseteq S_{u,t}^{-}$$
	and, consequently,
	$$\mathbb{P}(\Gamma\in S_{u,s_{K}(u)}^{-})\leq\mathbb{P}(\Gamma\in S_{u,t}^{-}).$$
	Analogously, by \eqref{d},
	$$\mathbb{P}(\Gamma\in S_{u,s_{K}(u)}^{+})\leq\mathbb{P}(\Gamma\in S_{u,t}^{+}).$$
	Taking the infimum in \eqref{definicionTukeyCompacto}, we can express $D_{CT}$ as
	\begin{equation*}
		D_{CT}(K;\Gamma) = \min\{\inf_{u\in\mathbb{S}^{p-1}}\mathbb{P}(\Gamma\in S_{u,s_{K}(u)}^{-}), \inf_{u\in\mathbb{S}^{p-1}}\mathbb{P}(\Gamma\in S_{u,s_{K}(u)}^{+})\}.
	\end{equation*}
	Making use of the definition of the halfspaces in \eqref{i} and \eqref{d}, we have
	\begin{equation}\label{expresionTukey2}
		D_{CT}(K;\Gamma) = \min\{\inf_{u\in\mathbb{S}^{p-1}}\mathbb{P}(s_{\Gamma}(u)\leq s_{K}(u)), \inf_{u\in\mathbb{S}^{p-1}}\mathbb{P}(s_{\Gamma}(u)\geq s_{K}(u))\},
	\end{equation}
	which coincides with the definition of the \textit{half-space depth} proposed by  \citet{Cascos}.
	
	Interchanging the minimum and  infimum in  \eqref{expresionTukey2},
	\begin{equation}\label{df}
		D_{CT}(K;\Gamma) = \inf_{u\in\mathbb{S}^{p-1}} \min\{\mathbb{P}(s_{\Gamma}(u)\leq s_{K}(u)),\mathbb{P}(s_{\Gamma}(u)\geq s_{K}(u))\}.
	\end{equation}
	Then, taking into account \eqref{HD}, we can express the Tukey depth for compact convex random sets in terms of the multivariate halfspace depth in the following way
	\begin{equation}\label{Tukeyhalfspace}
		D_{CT}(K;\Gamma) = \inf_{u\in\mathbb{S}^{p-1}} HD(s_{K}(u);s_{\Gamma}(u)).
	\end{equation}

\subsection{Sample halfspace depth}
We define the sample version $D_{CT,n}$ of the Tukey depth for compact convex sets. Let $$\Gamma : \Omega\rightarrow\mathcal{K}_{c}(\mathbb{R}^{p})$$ be a compact convex random set associated with the probabilistic space $(\Omega,\mathcal{A},\mathbb{P})$  and $X_{1},\ldots , X_{n}$ independent random sets distributed as $\Gamma$.
We define the sample version of the Tukey depth, $D_{CT,n}$ as
\begin{equation}\label{n}
	D_{CT,n}(K;\Gamma) := \min\{\inf_{u\in\mathbb{S}^{p-1}}\mathbb{P}_{n}^{u}((-\infty, s_{K}(u)]), \inf_{u\in\mathbb{S}^{p-1}}\mathbb{P}_{n}^{u}([s_{K}(u), \infty))\},
\end{equation}
for every $K\in\mathcal{K}_{c}(\mathbb{R}^{p})$, where
\begin{equation}
	\begin{aligned}\nonumber
		\mathbb{P}_{n}^{u}((-\infty, x]) &= \cfrac{1}{n}\cdot\sum_{i = 1}^{n}\text{I}(s_{X_{i}}(u)\in (-\infty,x]), \\ \nonumber
		\mathbb{P}_{n}^{u}([x, \infty)) &= \cfrac{1}{n}\cdot\sum_{i = 1}^{n}\text{I}(s_{X_{i}}(u)\in [x, \infty)),
	\end{aligned}
\end{equation}
for all $u\in\mathbb{S}^{p-1}$ and $x\in\mathbb{R}$. The function $D_{CT,n}$ coincides with the sample version of the \textit{half-space depth} proposed by \citet{Cascos}. Interchanging the minimum and infimum in  \eqref{n}, we also have that
\begin{equation}\label{n1}
	D_{CT,n}(K;\Gamma) := \inf_{u\in\mathbb{S}^{p-1}}\min\{\mathbb{P}_{n}^{u}((-\infty, s_{K}(u)]), \mathbb{P}_{n}^{u}([s_{K}(u), \infty))\}.
\end{equation}

\section{Properties of depth for compact convex sets}\label{properties}
In this section we propose some desirable properties for a depth for compact convex sets.
They are
mainly based on the properties that constitute the  notion  of depth function for multivariate spaces \citep{ZuoSerfling},   for  functional (metric) spaces  \citep{NietoBattey} and for the fuzzy setting \citep{primerarticulo}. Furthermore, we study whether $D_{CT}$ satisfies them.

Some of these properties parallel the ones considered in \cite{primerarticulo} and in certain cases they follow for a random set $\Gamma$ by applying the corresponding property in \cite{primerarticulo} to the indicator function $I_\Gamma$. However, this application is simplest for the properties whose direct proof is already very simple, which does not support the cost-effectiveness of doing so. In the longer proofs, additional arguments are needed due, for instance, to the subtlety that the deepest point in the (larger) space of fuzzy sets might conceivably be deeper than the deepest non-fuzzy set. Therefore, the properties referring to deepest points are parallel in wording but might potentially have different content. It can be proved that this does not actually happen; but we also found that direct proofs make the paper more self-contained. Thus we opted for proofs which do not require the reader to be familiar with the specifics of fuzzy sets, by adapting the arguments in \cite{primerarticulo}. Still, some other properties in this section were not considered in \cite{primerarticulo}.

\subsection{Property I. Affine invariance.} We focus on the M1. property of the multivariate case reported in the introduction. In the case of $\mathcal{K}_{c}(\mathbb{R}^{p})$, the product of $M\in\mathcal{M}_{n\times n}(\mathbb{R})$ times $K\in\mathcal{K}_{c}(\mathbb{R}^{p})$ is defined as the compact convex set
\begin{eqnarray}\label{M}
	M\cdot K = \{M\cdot k : k\in K\}.
\end{eqnarray}
The affine invariance property that we propose is the following.

\begin{itemize}
	\item[\bf (P1.)]  \textit{ Let $\Gamma$ be a compact convex random set, $D(\cdot;\Gamma) : \mathcal{K}_{c}(\mathbb{R}^{p})\rightarrow [0,\infty)$ a function. Then,
		\begin{equation}\nonumber
			D(M\cdot K + L; M\cdot\Gamma + L) = D(K;\Gamma),
		\end{equation}
		for all $M\in\mathcal{M}_{n\times n}(\mathbb{R})$ non-singular matrix and any $K,L\in\mathcal{K}_{c}(\mathbb{R}^{p})$.}
\end{itemize}

Thus, this property is analogous to the multivariate case. The property in the fuzzy case is different only in that we need the \textit{Zadeh's extension principle}  \cite{extesion} to apply a matrix to a fuzzy set. The property for functional data also differs, since \cite{NietoBattey} demands isometry invariance. However, note that in this context affine invariance actually implies isometry invariance since, by a result of \citet{Gruber}, all isometries of $\mathcal{K}_{c}(\mathbb{R}^{p})$ are of the form $K\mapsto M\cdot K + L$ with $M$ orthogonal.

\begin{proposition}\label{TukeyAffine}
	The function $D_{CT}$ satisfies P1.
\end{proposition}

The following lemma (cf. \cite[Proposition 8.2]{primerarticulo}) is used to prove Proposition \ref{TukeyAffine}.

\begin{lemma}\label{lemaesferas}
	Let $K\in\mathcal{K}_{c}(\mathbb{R}^{p})$ and $M\in\mathcal{M}_{n\times n}(\mathbb{R})$ a non-singular matrix. Then, $$s_{M\cdot K}(u) = \|M^{T}\cdot u\|\cdot s_{K}((1/(\|M^{T}\cdot u\|))\cdot M^{T}\cdot u)$$ for all $u\in\mathbb{S}^{p-1}$.
\end{lemma}

\begin{proof}
	Taking into account \eqref{M}, it is clear that
	\begin{equation}\nonumber
		s_{M\cdot K}(u) = \sup_{v\in M\cdot K} \langle u,v\rangle = \sup_{k\in K} \langle u, M\cdot k\rangle = \sup_{k\in K} \langle M^{T}\cdot u, k\rangle,
	\end{equation}
	for any $u\in\mathbb{S}^{p-1}$. In general, $M^{T}\cdot u$ does not belong to $\mathbb{S}^{p-1}$. Thus, normalizing it, we have that
	\begin{equation}
		\begin{aligned}\nonumber
			s_{M\cdot K}(u) = &\sup_{k\in K}\langle \|M^{T}\cdot u\|\cdot\cfrac{1}{\|M^{T}u\|}\cdot M^{T}\cdot u, k\rangle = \\ \nonumber
			& \|M^{T}\cdot u\|\cdot\sup_{k\in K}\langle \cfrac{1}{\|M^{T}\cdot u\|}\cdot M^{T}\cdot u,k\rangle = \|M^{T}\cdot u\|\cdot s_{K}(\cfrac{1}{\|M^{T}\cdot u\|}\cdot M^{T}\cdot u).
		\end{aligned}
	\end{equation}
\end{proof}

It is clear that, if $M\in\mathcal{M}_{n\times n}(\mathbb{R})$ is a non-singular matrix, the map $$f : \mathbb{S}^{p-1}\rightarrow\mathbb{S}^{p-1}$$ defined by $$f(u) = (1/\|M^{T}\cdot u\|)\cdot M^{T}\cdot u$$ is bijective. We make use of this to prove Proposition \ref{TukeyAffine}.

\begin{proof}[Proof of Proposition \ref{TukeyAffine}]
	Using properties of the support function and Lemma \ref{lemaesferas}, we get
	\begin{equation}\nonumber
		s_{M\cdot K + L}(u) = \|M^{T}\cdot u\|\cdot s_{K}(\cfrac{1}{\|M^{T}\cdot u\|}\cdot M^{T}\cdot u) + s_{L}(u),
	\end{equation}
	for all $u\in\mathbb{S}^{p-1}$. By  \eqref{expresionTukey2}, we have that
	\begin{equation}
		\begin{aligned}
			&\inf_{u\in\mathbb{S}^{p-1}}\mathbb{P}(s_{M\cdot\Gamma + L}(u)\leq s_{M\cdot K + L}(u)) = \inf_{u\in\mathbb{S}^{p-1}}\mathbb{P}(s_{M\cdot\Gamma}(u)\leq s_{M\cdot K}(u)) = \\ \nonumber
			&\inf_{u\in\mathbb{S}^{p-1}}\mathbb{P}(s_{\Gamma}(\cfrac{1}{\|M^{T}\cdot u\|}\cdot  M^{T}\cdot u)\leq s_{A}(\cfrac{1}{\|M^{T}\cdot u\|}\cdot M^{T}\cdot u)) = \\ \nonumber
			&\inf_{u\in\mathbb{S}^{p-1}}\mathbb{P}(s_{\Gamma}(u)\leq s_{K}(u))
		\end{aligned}
	\end{equation}
	where the last equality follows from the fact that $f$ is bijective.
\end{proof}

\subsection{Property II. Maximality at the center of symmetry.} In this case, the property is the same for multivariate, functional and fuzzy settings, but for the fact that the notion of symmetry applied has to be defined in the corresponding space.  In the multivariate case, there exist several  notions of symmetry, for instance central, angular and halfspace symmetry \citep{ZuoSerfling, LiuSimplicial}.  In the functional case, there exists one proved to be  topologically valid \citep{NietoBatteyJMVA, aplicacion1}, while there have been two proposals in the fuzzy setting \citep{primerarticulo}.
To propose a notion of symmetry in $\mathcal{K}_{c}(\mathbb{R}^{p}),$ we make use of the central symmetry notion and of the support function of compact convex random sets.
A random variable $X$ on $\mathbb{R}^{p}$ is \textit{centrally symmetric} (or $C$-symmetric) with respect to $x\in\mathbb{R}^{p}$ if $X - x $ and $x - X$ are equally distributed.

\begin{definition}
	Let $\Gamma$ be a compact convex random set. We say that $\Gamma$ is \textit{compact-symmetric} with respect to $K$ if $s_{\Gamma}(u)$ is $C$-symmetric with respect to $s_{K}(u)$ for all $u\in\mathbb{S}^{p-1}$.
\end{definition}

We propose the following property.

\begin{itemize}
	\item[\bf (P2.)]  \textit{ Let $\Gamma$ be a compact convex random set which is symmetric (for a certain notion of symmetry) with respect to $K\in\mathcal{K}_{c}(\mathbb{R}^{p})$. Let $D(\cdot;\Gamma) : \mathcal{K}_{c}(\mathbb{R}^{p})\rightarrow [0,\infty)$ be a function. Then
		\begin{equation}\nonumber
			D(K;\Gamma) = \sup_{L\in\mathcal{K}_{c}(\mathbb{R}^{p})} D(L;\Gamma).
		\end{equation}
	}
\end{itemize}

Thus, this property is analogous in the multivariate, functional and fuzzy case. The only difference is the notion of symmetry defined for each case. Note that the above defined notion of symmetry for compact convex random sets, which makes use of $C$-symmetry, is also an adaptation of the $F$-symmetry \citep{primerarticulo} of the fuzzy case, based on support functions.

With the above notion of compact-symmetry, we have the following result.

\begin{proposition}\label{TukeyP2}
	The function $D_{CT}$ satisfies P2.
\end{proposition}

\begin{proof}
	By hypothesis, let us suppose that $\Gamma$ is compact-symmetric with respect to $K$. By definition, we have that the real random variable $s_{\Gamma}(u)$ is $C$-symmetric with respect to $s_{K}(u)$ for all $u\in\mathbb{S}^{p-1}$. This means that $$s_{K}(u) \in \mbox{Med}(s_{\Gamma}(u))$$ for all $u\in\mathbb{S}^{p-1}$, where $ \mbox{Med}(\cdot)$ denotes the univariate median. It implies that $$\mathbb{P}(s_{\Gamma}(u)\leq s_{K}(u))\geq 1/2\mbox{ and }\mathbb{P}(s_{\Gamma}(u)\geq s_{K}(u))\geq 1/2.$$ Using the expression of $D_{CT}$ in Equation \ref{expresionTukey2}, we have that $D_{CT}(\cdot;\Gamma)$ is maximized in $K$.
\end{proof}

\subsection{Property III. Monotonicity with respect to the center.} In the multivariate case  \cite{ZuoSerfling}, this property is understood in an algebraic way, as the convex combinations between the element of maximal depth and another point are considered. As the operations of sum and  product by a scalar are defined in $\mathcal{K}_{c}(\mathbb{R}^{p})$ , we can propose the same property.

\begin{itemize}
	\item[\bf (P3a.)]  \textit{ \mbox{ } Let $\Gamma$ be a compact convex random set and let $K\in\mathcal{K}_{c}(\mathbb{R}^{p})$ maximize $D(\cdot;\Gamma)$. Then, 
		$$D((1-\lambda)\cdot K + \lambda\cdot L;\Gamma)\geq D(L;\Gamma)$$ 
		for all $\lambda\in [0,1]$ and $L\in\mathcal{K}_{c}(\mathbb{R}^{p})$.
	}
\end{itemize}

Additionally, this property is analogous to property P3a. in the definition of semilinear depth in the fuzzy setting \cite{primerarticulo}.

In the functional (metric) case, a different property was proposed by \citet[Property P-3.]{NietoBattey} which uses explicitly the metric in the space. We can see $\mathcal{K}_{c}(\mathbb{R}^{p})$ as a metric space with the Hausdorff metric $d_{\mathcal{H}}$. Thus, another possible property is the following.

\begin{itemize}
	\item[\bf (P3b.)]  \textit{ \mbox{ }  Let $\Gamma$ be a compact convex random set, $d$ a metric in $\mathcal{K}_{c}(\mathbb{R}^{p})$ and $K,L,S\in\mathcal{K}_{c}(\mathbb{R}^{p})$ three sets such that $K$ maximizes $D(\cdot;\Gamma)$ and $d(K,S) = d(K,L) + d(L,S)$. Then, $$D(L;\Gamma)\geq D(S;\Gamma).$$
	}
\end{itemize}

This property is analogous to property P3b. in the definition of geometric depth in the fuzzy setting \cite{primerarticulo}.

About these two possible translations of the multivariate property, we have the following two results.

\begin{proposition}\label{TukeyP3a}
	The function $D_{CT}$ satisfies P3a.
\end{proposition}

\begin{proof}
	Let $\Gamma$ be a compact convex random set and let $K,L\in\mathcal{K}_{c}(\mathbb{R}^{p})$ be two sets such that $K$ maximizes $D_{CT}(\cdot;\Gamma)$. Using the properties of the support function of a set, we have that $$s_{(1 - \lambda)\cdot K + \lambda\cdot L} (u) = (1-\lambda) s_{K}(u) + \lambda s_{L}(u)$$ for all $u\in\mathbb{S}^{p-1}$ and $\lambda\in [0,1]$.
	
	We consider the set $$\mathbb{K} = \{(u,t)\in\mathbb{S}^{p-1}\times\mathbb{R} : (1-\lambda)\cdot K + \lambda\cdot L\in S_{u,t}^{-}\}.$$ It can be expressed as $\mathbb{K}_{1}\cup\mathbb{K}_{2}\cup\mathbb{K}_{3}$, where
	\begin{equation}
		\begin{aligned}\nonumber
			\mathbb{K}_{1} &= \{(u,t)\in\mathbb{S}^{p-1}\times\mathbb{R} : K, L\in S_{u,t}^{-}, L\in S_{u,t}^{-}\}, \\ \nonumber
			\mathbb{K}_{2} &= \{(u,t)\in\mathbb{S}^{p-1}\times\mathbb{R} : K\in S_{u,t}^{-}, L\not\in S_{u,t}^{-}, (1-\lambda)\cdot K + \lambda\cdot L\in S_{u,t}^{-}\}, \\ \nonumber
			\mathbb{K}_{3} &= \{(u,t)\in\mathbb{S}^{p-1}\times\mathbb{R} : K\not\in S_{u,t}^{-}, L\in S_{u,t}^{-}, (1-\lambda)\cdot K + \lambda\cdot L\in S_{u,t}^{-}\}.
		\end{aligned}
	\end{equation}
	It is clear that they are disjoint sets. Thus, we have that
	\begin{equation}\label{ecuacion1Tukey3}
		\begin{aligned}
			&\inf_{\underset{(1-\lambda)\cdot K + \lambda\cdot L\in S_{u,t}^{-}}{u\in\mathbb{S}^{p-1}, t\in\mathbb{R} :}}\mathbb{P}(\Gamma\in S_{u,t}^{-}) = \inf_{(u,t)\in\mathbb{K}}\mathbb{P}(\Gamma\in S_{u,t}^{-}) = \\
			&\min\{\inf_{(u,t)\in\mathbb{K}_{1}}\mathbb{P}(\Gamma\in S_{u,t}^{-}), \inf_{(u,t)\in\mathbb{K}_{2}}\mathbb{P}(\Gamma\in S_{u,t}^{-}), \inf_{(u,t)\in\mathbb{K}_{3}}\mathbb{P}(\Gamma\in S_{u,t}^{-})\}.
		\end{aligned}
	\end{equation}
	
	Taking into account that $$\mathbb{K}_{1},\mathbb{K}_{2}\subseteq\{(u,t)\in\mathbb{S}^{p-1}\times\mathbb{R} : K\in S_{u,t}^{-}\}$$ and $$\mathbb{K}_{3}\subseteq\{(u,t)\in\mathbb{S}^{p-1}\times\mathbb{R} : L\in S_{u,t}^{-}\},$$ it is obtained that
	\begin{equation}\label{ecuacion2Tukey3}
		\begin{aligned}
			\inf_{(u,t)\in\mathbb{K}_{1}}\mathbb{P}(\Gamma\in S_{u,t}^{-})&\geq\inf_{\underset{K\in S_{u,t}^{-}}{u\in\mathbb{S}^{p-1}, t\in\mathbb{R}:}}\mathbb{P}(\Gamma\in S_{u,t}^{-})\geq D_{CT}(K;\Gamma),\\
			\inf_{(u,t)\in\mathbb{K}_{2}}\mathbb{P}(\Gamma\in S_{u,t}^{-})&\geq\inf_{\underset{K\in S_{u,t}^{-}}{u\in\mathbb{S}^{p-1}, t\in\mathbb{R}:}}\mathbb{P}(\Gamma\in S_{u,t}^{-})\geq D_{CT}(K;\Gamma),\\
			\inf_{(u,t)\in\mathbb{K}_{3}}\mathbb{P}(\Gamma\in S_{u,t}^{-})&\geq\inf_{\underset{L\in S_{u,t}^{-}}{u\in\mathbb{S}^{p-1}, t\in\mathbb{R}:}}\mathbb{P}(\Gamma\in S_{u,t}^{-})\geq D_{CT}(L;\Gamma).
		\end{aligned}
	\end{equation}
	Using  \eqref{ecuacion1Tukey3} and \eqref{ecuacion2Tukey3} and taking into account that $K$ maximizes $D_{CT}$, we have that
	\begin{equation}\nonumber
		\inf_{\underset{(1-\lambda)\cdot K + \lambda\cdot L\in S_{u,t}^{-}}{u\in\mathbb{S}^{p-1}, t\in\mathbb{R} :}}\mathbb{P}(\Gamma\in S_{u,t}^{-})\geq D_{CT}(L;\Gamma).
	\end{equation}
	Analogously, we get
	\begin{equation}\nonumber
		\inf_{\underset{(1-\lambda)\cdot K + \lambda\cdot L\in S_{u,t}^{+}}{u\in\mathbb{S}^{p-1}, t\in\mathbb{R} :}}\mathbb{P}(\Gamma\in S_{u,t}^{+})\geq D_{CT}(L;\Gamma).
	\end{equation}
	Thus, $D_{CT}((1-\lambda)\cdot K + \lambda L;\Gamma)\geq D_{CT}(L;\Gamma)$ and $D_{CT}$ satisfies property P3a.
\end{proof}

\begin{proposition}\label{TukeyP3b}
	The function $D_{CT}$ does not satisfy P3b with respect to the distance $d_{\mathcal{H}}$.
\end{proposition}

\begin{proof}
	The proof is by counterexample. Let $(\{\omega_{1},\omega_{2}\},\mathcal{P}(\{\omega_{1},\omega_{2}\}), \mathbb{P})$ be a probabilistic space such that $$\mathbb{P}(\omega_{1}) = 3/4\mbox{ and }\mathbb{P}(\omega_{2}) = 1/4.$$ We consider the compact convex random set $\Gamma : \Omega\rightarrow\mathcal{K}_{c}(\mathbb{R})$ defined by $$\Gamma(\omega_{1}) = [1,2]\mbox{ and }\Gamma(\omega_{2}) = [2,7].$$ It is clear that $$D_{CT}(\Gamma(\omega_{1});\Gamma) = 3/4$$ and it is the set which maximizes $D_{CT}$. Let us consider $L = [3,5]$. We have that
	\begin{equation}\nonumber
		5 = d_{\mathcal{H}}(\Gamma(\omega_{1}), \Gamma(\omega_{2})) = d_{\mathcal{H}}(\Gamma(\omega_{1}),L) + d_{\mathcal{H}}(\Gamma(\omega_{2}),L) = 3 + 2.
	\end{equation}
	Moreover, $$D_{CT}(\Gamma(\omega_{2});\Gamma) = 1/4\mbox{ and }D_{CT}(L;\Gamma) = \mathbb{P}(s_{\Gamma}(-1)\leq s_{C}(-1)) = 0.$$ Thus, $D_{CT}$ violates property P3b.
\end{proof}

Notice that Tukey depth may satisfy Property P3b if distances between sets are not measured with the Hausdorff metric, e.g., in the $L^p$-type metrics introduced by \citet{Vitale}.

\subsection{Property IV. Vanishing at infinity.} The property in the multivariate case is understood in a geometrical way, considering a sequence $\{x_{n}\}_{n}$ such that $\|x_{n}\|\rightarrow\infty$ \cite{ZuoSerfling}. We can also consider a sequence $\{a + nb\}_{n}$ with $a,b\in\mathbb{R}^{p}$ such that $b\neq 0$ and suppose that the sequence of distances diverges. Thus, on this setting, we also propose two possible properties, the first one from an algebraic point of view and the second one taking into account that the set $\mathcal{K}_{c}(\mathbb{R}^{p})$ can be viewed as a metric space using the Hausdorff distance.

\begin{itemize}
	\item[\bf (P4a.)]  \textit{ \mbox{ } Let $\Gamma$ be a compact convex random set and let $K,L\in\mathcal{K}_{c}(\mathbb{R}^{p})$ be two sets such that $K$ maximizes $D(\cdot;\Gamma)$ and $L\neq\{0\}$. Then, $$\lim_{n} D(K+n\cdot L;\Gamma) = 0.$$}
	
	\item[\bf (P4b.)]  \textit{ \mbox{ } Let $\Gamma$ be a compact convex random set, $d$ a metric in $\mathcal{K}_{c}(\mathbb{R}^{p})$, $K\in\mathcal{K}_{c}(\mathbb{R}^{p})$  a set  that maximizes $D(\cdot;\Gamma)$ and $\{K_{n}\}_{n}$ a sequence of elements of $\mathcal{K}_{c}(\mathbb{R}^{p})$ such that $\lim_{n} d(K,K_{n}) = \infty$. Then,
		$$\lim_{n} D(K_{n};\Gamma) = 0.$$
	}
\end{itemize}
Property P4a. parallels the fourth property of the semilinear depth for fuzzy sets, while  P4b. parallels the fourth property of geometric depth for fuzzy sets.

About those properties we have the following results.

\begin{proposition}\label{TukeyP4}
	The function $D_{CT}$ satisfies P4a. and P4b. with respect to the distance $d_{\mathcal{H}}$.
\end{proposition}

The following proposition is used in the proof of Proposition \ref{TukeyP4} for property P4b.

\begin{proposition}\label{propositionHausdorff}
	Let $\{K_{n}\}_{n}$ be a sequence of elements of $\mathcal{K}_{c}(\mathbb{R}^{p})$ such that $\lim_{n} d_{\mathcal{H}}(K_{n};\{0\}) = \infty$. Then, there exists $u\in\mathbb{S}^{p-1}$ such that $$\lim_{n} s_{K_{n}}(u) = \infty.$$
\end{proposition}

\begin{proof}
	It is a basic property of the Hausdorff distance that $$d_{\mathcal{H}}(K_{n},\{0\}) = \sup\{\|x\| : x\in K_{n}\}$$ for all $n\in\mathbb{N}$. The function $$f_{n} : K_{n}\rightarrow\mathbb{R}$$ defined by $$f_{n}(x) = \|x\|$$ is a continuous function defined over a compact convex set, thus it attains its maximum on $K_{n}$, for all $n\in\mathbb{N}$. Let us denote by $x_{n}$ the point of $K_{n}$ where $f_{n}$ attains its maximum for every $n\in\mathbb{N}$. By hypothesis we have that $$\lim_{n} \|x_{n}\| = \infty.$$ It implies that there exists $u\in\mathbb{S}^{p-1}$ such that $$\lim_{n}\langle u,x_{n}\rangle = \infty.$$ By definition of the support function of a compact convex set, we have that $$\langle u,x_{n}\rangle\leq s_{K_{n}}(u).$$ Thus, $\lim_{n} s_{K_{n}}(u) = \infty$.
\end{proof}

\begin{proof}[Proof of Proposition \ref{TukeyP4}]
	\ \\
	\textbf{Property P4a.} Let $L\neq\{ 0\}$. There exists $u_{0}\in\mathbb{S}^{p-1}$ such that $$s_{L}(u_{0})\neq 0.$$ Without loss of generality we assume $s_{L}(u_{0}) > 0$. Clearly, the sequence $$\{s_{K}(u_{0}) + n\cdot s_{L}(u_{0})\}_{n}$$ is such that $$\lim_{n} s_{K}(u_{0}) + n\cdot s_{L}(u_{0}) = \infty.$$ We have that
	\begin{equation}\nonumber
		D_{CT}(K+\cdot L;\Gamma)\leq \mathbb{P}(s_{\Gamma}(u_{0})\geq s_{K}(u_{0}) + n\cdot s_{L}(u_{0})).
	\end{equation}
	If we take limits in both sides
	\begin{equation}\nonumber
		\lim_{n\rightarrow\infty} D_{CT}(K+n\cdot L;\Gamma)\leq \lim_{n\rightarrow\infty}\mathbb{P}(s_{\Gamma}(u_{0})\geq s_{K}(u_{0}) + n\cdot s_{L}(u_{0})) = 0.
	\end{equation}
	Using the Sandwhich's Rule we have that $\lim_{n} D_{CT}(K + n\cdot L;\Gamma) = 0$.
	
	\vspace{.3cm}
	\textbf{Property P4b.} As the set $K$ is fixed, condition $$\lim_{n} d_{\mathcal{H}}(K,K_{n}) = \infty$$ is equivalent to $$\lim_{n}d_{\mathcal{H}}(K_{n},\{0\}) = \infty.$$ By Proposition \ref{propositionHausdorff}, we have that there exists $u_{0}\in\mathbb{S}^{p-1}$ such that $$\lim_{n} s_{K_{n}}(u) = \infty.$$ The rest of the proof is analogous to that of Property P4a.
\end{proof}

\subsection{Property V. Upper semicontinuity.} This property regards a depth as an upper semicontinuous function at every point of its domain. In the multivariate case it is not considered to be a canonical requirement, but continuity properties are studied in different papers, for instance in  \cite{ZuoSerfling}. This property  is considered in the definition of depth function for functional (metric) spaces  \cite{NietoBattey}.
According to  \cite{NietoBattey},  a depth $D,$  on a metric space $(\mathbb{E},d)$  with respect to a distribution $\mathbb{P}$ on the space, is upper semicontinuous if for all $x\in\mathbb{E}$ and for all $\varepsilon > 0$, there exists $\delta > 0$ such that
$$
\sup_{y : d(x,y)<\delta} D(y;\mathbb{P}) \leq D (x;\mathbb{P}).
$$
The property has not yet being considered in the fuzzy setting. 
\begin{itemize}
	\item[\bf (P5.)]  \textit{ Let $\Gamma$ be a compact convex random set and $d$ a metric defined over $\mathcal{K}_{c}(\mathbb{R}^{p})$. The function $D(\cdot;\Gamma)$ is upper semicontinuous with respect to the distance $d$  in the sense that $$
		\lim\sup_{n} D(K_{n};\Gamma)\leq D(K;\Gamma)
		$$ 
		for every set $K\in\mathcal{K}_{c}(\mathbb{R}^{p})$ and every sequence of sets $\{K_{n}\}_{n}$ such that  $lim_{n}d(K,K_{n}) = 0.$}
\end{itemize}

Notice that upper semicontinuity implies that the contours of the depth function are closed sets.

\begin{proposition}\label{proposicionupperTukey}\label{muj}
	The function $D_{CT}$ satisfies P5. with respect to the distance $d_{\mathcal{H}}$.
\end{proposition}

\begin{proof}
	Let $\Gamma$ be a compact convex random set, $K\in\mathcal{K}_{c}(\mathbb{R}^{p})$ a set and let $\{K_{n}\}_{n}$ be a sequence of compact convex sets such that $$\lim_{n\rightarrow\infty}d_{\mathcal{H}}(K,K_{n}) = 0.$$ We need to prove
	$$\lim\sup_{n}D_{CT}(K_{n};\Gamma)\leq D_{CT}(K;\Gamma).$$
	By \eqref{ecuacionHausdorff},
	$$d_{\mathcal{H}}(K,K_{n}) = \sup_{u} |s_{K}(u) - s_{K_{n}}(u)|$$
	and then
	$$\lim_{n\rightarrow\infty}|s_{K}(u) - s_{K_{n}}(u)| = 0$$
	for each $u\in\mathbb{S}^{p-1}$. Thus $$\lim_{n\rightarrow\infty}s_{K_{n}}(u) = s_{K}(u)$$ for every $u\in\mathbb{S}^{p-1}$.
	Without loss of generality (the other case is analogous), assume $$D_{CT}(K;\Gamma) = \inf_{u}\mathbb{P}(s_{K}(u)\leq s_{\Gamma}(u)).$$
	Now, we prove that, for all $u\in\mathbb{S}^{p-1}$,
	\begin{equation}\nonumber
		U := \{\omega\in\Omega : \forall k\in\mathbb N, \exists n\geq k : \omega\in\{s_{K_{n}}(u)\leq s_{\Gamma}(u)\}\}\subseteq\{\omega\in\Omega : s_{K}(u)\leq s_{\Gamma(\omega)}(u)\}.
	\end{equation}
	Let $\omega\in U$. There exists a subsequence $\{K_{n'}\}_{n}$ of $\{K_{n}\}_{n}$ such that $$s_{K_{n'}}(u)\leq s_{\Gamma(\omega)}(u)$$ for all $n'$. Taking limits,
	$$s_{K}(u)=\lim_{n'\rightarrow\infty}s_{K_{n'}}(u) \leq s_{\Gamma(\omega)}(u),$$
	therefore
	$$\omega\in\{\omega\in\Omega : s_{K}(u)\leq s_{\Gamma}(u)\}.$$
	By definition, $U=\limsup_{n}\{s_{K_{n}}(u)\leq s_{\Gamma}(u)\}$. Thus
	\begin{equation}\label{ec1TukeyCompacto}
		\begin{aligned}
			D_{CT}(K;\Gamma) &= \mathbb{P}(s_{K}(u)\leq s_{\Gamma}(u))\geq\mathbb{P}(\limsup_{n}\{s_{K_{n}}(u)\leq s_{\Gamma}(u)\})\geq \\
			&\geq\limsup_{n\rightarrow\infty}\mathbb{P}(s_{K_{n}}(u)\leq s_{\Gamma}(u))
		\end{aligned}
	\end{equation}
	where the second inequality is due to the  Fatou's lemma. Taking the infimum in both sides yields
	\begin{equation}\label{ec2TukeyCompacto}
		\inf_{u}\mathbb{P}(s_{K}(u)\leq s_{\Gamma}(u))\geq\inf_{u}\lim\sup_{n\rightarrow\infty}\mathbb{P}(s_{K_{n}}(u)\leq s_{\Gamma}(u)).
	\end{equation}
	Since
	$$
	\lim\sup_{n}\mathbb{P}(s_{K_{n}}(u)\leq s_{\Gamma}(u)) = \inf_{n}\sup_{k\geq n}\mathbb{P}(s_{K_{k}}(u)\leq s_{\Gamma}(u)),
	$$
	it is clear that
	\begin{equation}\label{ec3TukeyCompacto}
		\begin{aligned}
			&\inf_{u}\inf_{n}\sup_{k\geq n}\mathbb{P}(s_{K_{k}}(u)\leq s_{\Gamma}(u))\geq \inf_{n}\sup_{k\geq n}\inf_{u}\mathbb{P}(s_{K_{k}}(u)\leq s_{\Gamma}(u)) = \\
			&= \lim\sup_{n\rightarrow\infty} \inf_{u}\mathbb{P}(s_{K}(u)\leq s_{\Gamma}(u))\geq \lim\sup_{n\rightarrow\infty} D_{CT}(K_{n};\Gamma).
		\end{aligned}
	\end{equation}
	By  \eqref{ec1TukeyCompacto}, \eqref{ec2TukeyCompacto} and \eqref{ec3TukeyCompacto},  $D_{CT}(\cdot;\Gamma)$ is upper semicontinuous.
\end{proof}

\subsection{Property VI. Consistency.
} Another desirable property for depth functions is that the sample version converges to the population counterpart (consistency).
This property is a particular case of the weak continuity (as a function of the distribution $P$) property of the axiomatic functional (metric) notion of depth \cite{NietoBattey} but it is not part of the axiomatic notions of multivariate and fuzzy depth. However, it is generally studied when an instance of depth function is introduced. To the best of our knowledge, the first time that appeared in the literature for depth functions was in \citet{LiuSimplicial}.

We propose the following property.

\begin{itemize}
	\item[\bf (P6.)]  \textit{ Let $\Gamma$ be a compact convex random set, $D(\cdot;\Gamma) : \mathcal{K}_{c}(\mathbb{R}^{p})\rightarrow [0,\infty)$ a function and $D_{n}(\cdot;\Gamma)\mathcal{K}_{c}(\mathbb{R}^{p})\rightarrow [0,\infty)$ its sample version. Then, $D$ and $D_{n}$ satsify
		$$
		\sup_{K\in\mathcal{K}_{c}(\mathbb{R}^{p})} |D(K;\Gamma) - D_{n}(K;\Gamma)| \longrightarrow 0 \text{, a.s. } [\mathbb{P}].
		$$
	}
\end{itemize}

This is a uniform consistency requirement which is satisfied by the Tukey depth but the uniformity may eventually have to be dropped for other depth functions.

\begin{theorem}\label{teoremaconsistencia}
	The function $D_{CT},$ with $D_{CT,n}$ in \eqref{n}, satisfies P6.
\end{theorem}

\begin{proof}
	By measurability, we have that $s_{X_{1}}(u),\cdots , s_{X_{n}}(u)$ is a random sample of the random variable $s_{\Gamma}(u)$ for all $u\in\mathbb{S}^{p-1}$. Let us fix $K\in\mathcal{K}_{c}(\mathbb{R}^{p})$. 
	To ease the notation, let us denote $$F(s_{K}(u)):=\{\mathbb{P}(s_{\Gamma}(u)\leq s_{K}(u)), \mathbb{P}(s_{\Gamma}(u)\geq s_{K}(u))\},$$
	$$F_n(s_{K}(u)):=\{\mathbb{P}_{n}^{u}((-\infty, s_{K}(u)]) , \mathbb{P}_{n}^{u}([s_{K}(u), \infty))\}.$$
	By \eqref{df} and \eqref{n1} and basic properties of the supremum and infimum functions, we have that
	\begin{equation}
		\begin{aligned}\nonumber
			|D_{CT}(K;\Gamma) - D_{CT,n}(K;\Gamma)| &= |\inf_{u\in\mathbb{S}^{p-1}} \min F(s_{K}(u))- \inf_{u\in\mathbb{S}^{p-1}} \min F_n(s_{K}(u))|  \\ \nonumber
			& \leq
			\sup_{u\in\mathbb{S}^{p-1}} |\min F(s_{K}(u)) - \min F_n(s_{K}(u))|.
		\end{aligned}
	\end{equation}
	{\it Step 1.} Setting
	\begin{equation}
		\begin{aligned}\nonumber
			F^+(t,u)&:=\mathbb{P}(s_{\Gamma}(u)\leq t)
			\\ \nonumber
			F^-(t,u)&:=\mathbb{P}(s_{\Gamma}(u)\geq t)
			\\ \nonumber
			F_n^+(t,u)&:=\mathbb{P}_{n}^{u}((-\infty, t])
			\\ \nonumber
			F_n^-(t,u)&:=\mathbb{P}_{n}^{u}([t, \infty))
		\end{aligned}
	\end{equation}
	and applying again basic properties, we obtain
	
	\begin{equation}
		\begin{aligned}
			|D_{CT}(K;\Gamma) - D_{CT,n}(K;\Gamma)| \leq
			&\sup_{u\in\mathbb{S}^{p-1}} \max\{|F^+(s_{K}(u),u) - F^+_n(s_{K}(u),u)|,\\ \nonumber
			&\hspace{1.8cm} |F^-(s_{K}(u),u) -F^-_n(s_{K}(u),u)|\}.
		\end{aligned}
	\end{equation}
	Then
	\begin{equation}
		\begin{aligned}\nonumber
			\sup_{K\in\mathcal{K}_{c}(\mathbb{R}^{p})} & |D_{CT}(K;\Gamma) - D_{CT,n}(K;\Gamma)|\leq 	 \\ \nonumber
			&\leq \sup_{K\in\mathcal{K}_{c}(\mathbb{R}^{p})} \sup_{u\in\mathbb{S}^{p-1}} \max\{|F^+(s_{K}(u),u) - F^+_n(s_{K}(u),u)|, \\ \nonumber
			& \hspace{3.7cm} |F^-(s_{K}(u),u) -F^-_n(s_{K}(u),u)|\}
			\\ \nonumber
			&\leq \sup_{u\in\mathbb{S}^{p-1}}\sup_{t\in \mathbb{R}} \max\{|F^+(t,u) - F^+_n(t,u)|, |F^-(t,u) -F^-_n(t,u)|\}.
		\end{aligned}
	\end{equation}
	The Dvoretzky--Kiefer--Wolfowitz inequality \cite[Corollary 1]{Massart} gives, for each $u\in\mathbb{S}^{p-1}$ and $\varepsilon>0$,
	$$P(\sup_{t\in \mathbb{R}} |F^+(t,u) - F^+_n(t,u)|>\varepsilon)\le 2\exp\{-2\varepsilon^2n\}$$
	and there easily follows
	$$P(\sup_{t\in \mathbb{R}} |F^-(t,u) - F^-_n(t,u)|>\varepsilon)\le 2\exp\{-2\varepsilon^2n\}.$$
	Since the bound is independent of $u$, that implies
	$$P(\sup_{u\in\mathbb{S}^{p-1}}\sup_{t\in \mathbb{R}} \max\{|F^+(t,u) - F^+_n(t,u)|, |F^-(t,u) -F^-_n(t,u)|\}>\varepsilon)\le 4\exp\{-2\varepsilon^2n\}$$
	which, by the arbitrariness of $\varepsilon$, establishes
	$$\sup_{u\in\mathbb{S}^{p-1}}\sup_{t\in \mathbb{R}} \max\{|F^+(t,u) - F^+_n(t,u)|, |F^-(t,u) -F^-_n(t,u)|\}\to 0$$
	in probability.
	
	{\it Step 2.} To prove almost sure convergence, we rewrite the supremum in terms of an empirical process. Taking
	$$\mathcal F=\{\phi_{t,u}^+,\phi_{t,u}^-\mid (t,u)\in\mathbb{R}\times \mathbb{S}^{p-1}\}$$
	where $\phi_{t,u}^+,\phi_{t,u}^-:\Omega\to\mathbb{R}$ are given by
	$$\phi_{t,u}^+(\omega)=I_{(-\infty,t](s_{\Gamma(\omega)})},\quad \phi_{t,u}^-(\omega)=I_{[t,\infty)(s_{\Gamma(\omega)})},$$
	we have
	$$\sup_{u\in\mathbb{S}^{p-1}}\sup_{t\in \mathbb{R}} \max\{|F^+(t,u) - F^+_n(t,u)|, |F^-(t,u) -F^-_n(t,u)|\}=\sup_{\phi\in\mathcal F}|E_{P_n}(\phi)-E_P(\phi)|$$
	where $P_n$ is the empirical distribution.
	By \cite[Corollary 3.7.9]{Nickl}, that supremum converges to 0 almost surely because it does so in probability (which was proved in Step 1) and the family $\mathcal F$ has a $P$-integrable measurable envelope, which is obvious since all functions in $\mathcal F$ take on values in $[0,1]$. Accordingly, also
	$$
	\sup_{K\in\mathcal{K}_{c}(\mathbb{R}^{p})} |D_{CT}(K;\Gamma) - D_{CT,n}(K;\Gamma)|\longrightarrow 0 \text{, a.s. } [\mathbb{P}].
	$$
\end{proof}

\subsection{Property VII. Convexity of the contours.
}
This property is not part of any of the existing axiomatic notions of statistical depth. However, it is commonly studied in the literature since it first appeared in \citet{Donoho}. In addition, \citet{Ser}, which focuses on multivariate properties,  lists it as a desirable property.

The set $\mathcal{K}_{c}(\mathbb{R}^{p})$ is endowed with the operations sum and product by a scalar. Thus, given $U\subseteq\mathcal{K}_{c}(\mathbb{R}^{p})$, we can say that $U$ is a convex set if $$(1-\lambda)\cdot K + \lambda\cdot L\in U$$ for every pair of sets $K,L\in U$ and for all $\lambda\in [0,1]$.
We propose the following property.

\begin{itemize}
	\item[\bf (P7.)]  \textit{ Let $\Gamma$ be a compact convex random set and  $D(\cdot;\Gamma) : \mathcal{K}_{c}(\mathbb{R}^{p})\rightarrow [0,\infty)$ a function. Then, the set
		$$
		D_{\alpha} := \{K\in\mathcal{K}_{c}(\mathbb{R}^{p}) : D(K;\Gamma)\geq\alpha\}\subseteq\mathcal{K}_{c}(\mathbb{R}^{p})
		$$
		is convex for every $\alpha\in [0,1]$.
	}
\end{itemize}

The next result states that the function $D_{CT}$ satisfies the above property, that is, the $\alpha$-contours of $D_{CT}$ are convex subsets of $\mathcal{K}_{c}(\mathbb{R}^{p})$.

\begin{theorem}\label{teoremaconvexidad}
	The function $D_{CT}$ satisfies P7.
\end{theorem}

\begin{proof}
	Let us fix $\alpha\in [0,1]$, $K,L\in D_\alpha$ and $\lambda\in [0,1]$. The aim is to prove
	$$(1-\lambda)\cdot K + \lambda\cdot L\in D_\alpha.$$ For that, we follow the same idea of the proof of Proposition \ref{TukeyP3a}. By the definition of Tukey depth,
	$$
	D_{CT}((1-\lambda)\cdot K + \lambda\cdot L;\Gamma) = \min\{\inf_{\underset{(1-\lambda)\cdot K + \lambda\cdot L\in S_{u,t}^{-}}{u\in\mathbb{S}^{p-1}, t\in\mathbb{R} :}} \mathbb{P}(\Gamma\in S_{u,t}^{-}),\inf_{\underset{(1-\lambda)\cdot K + \lambda\cdot L\in S_{u,t}^{+}}{u\in\mathbb{S}^{p-1}, t\in\mathbb{R} :}} \mathbb{P}(\Gamma\in S_{u,t}^{+})\}.
	$$
	We now prove that
	$$\inf_{\underset{(1-\lambda)\cdot K + \lambda\cdot L\in S_{u,t}^{-}}{u\in\mathbb{S}^{p-1}, t\in\mathbb{R} :}} \mathbb{P}(\Gamma\in S_{u,t}^{-})\geq\alpha.
	$$
	As in the proof of Proposition \ref{TukeyP3a}, we define the following sets
	\begin{equation}
		\begin{aligned}\nonumber
			\mathbb{K} &:= \{(u,t)\in\mathbb{S}^{p-1}\times\mathbb{R} : (1-\lambda)\cdot K + \lambda\cdot K\in S_{u,t}^{-}\},\\ \nonumber
			\mathbb{K}_{1} &:= \{(u,t)\in\mathbb{S}^{p-1}\times\mathbb{R} : K, L\in S_{u,t}^{-}, L\in S_{u,t}^{-}\}, \\ \nonumber
			\mathbb{K}_{2} &:= \{(u,t)\in\mathbb{S}^{p-1}\times\mathbb{R} : K\in S_{u,t}^{-}, L\not\in S_{u,t}^{-}, (1-\lambda)\cdot K + \lambda\cdot L\in S_{u,t}^{-}\}, \\ \nonumber
			\mathbb{K}_{3} &:= \{(u,t)\in\mathbb{S}^{p-1}\times\mathbb{R} : K\not\in S_{u,t}^{-}, L\in S_{u,t}^{-}, (1-\lambda)\cdot K + \lambda\cdot L\in S_{u,t}^{-}\}.
		\end{aligned}
	\end{equation}
	It is clear that
	$$
	\inf_{(u,t)\in\mathbb{K}}\mathbb{P}(\Gamma\in S_{u,t}^{-}) = \min\{\inf_{(u,t)\in\mathbb{K}_{1}}\mathbb{P}(\Gamma\in S_{u,t}^{-}),\inf_{(u,t)\in\mathbb{K}_{2}}\mathbb{P}(\Gamma\in S_{u,t}^{-}),\inf_{(u,t)\in\mathbb{K}_{3}}\mathbb{P}(\Gamma\in S_{u,t}^{-}) \}.
	$$
	Taking into account  \eqref{ecuacion2Tukey3} and the fact that $D_{CT}(K;\Gamma),D_{CT}(L;\Gamma)\geq\alpha$, we have that
	$$
	\inf_{(u,t)\in\mathbb{K}_{i}}\mathbb{P}(\Gamma\in S_{u,t}^{-})\geq \alpha
	$$
	for every $i\in\{1,2,3\}$. The case with $S_{u,t}^{+}$ is done analogously. Thus, $$D_{CT}((1-\lambda)\cdot K + \lambda\cdot L;\Gamma)\geq\alpha$$ and $D_{CT}(\cdot ;\Gamma)_{\alpha}$ is a convex set.
\end{proof}

\section{Discussion}\label{remarks}
Considering the properties studied in the literature for depth functions, we propose nine different properties for depth functions with respect to  compact convex random sets. They are:
\begin{itemize}
	\item P1. Affine invariance,
	\item P2. Maximality at the center of symmetry,
	\item P3a. Monotonicity with respect to the center in an algebraic way,
	\item P3b. Monotonicity with respect to the center in relation to the associated distance (in a geometric way),
	\item P4a. Vanishing at infinity in an algebraic way,
	\item P4b. Vanishing at infinity in a geometric way,
	\item P5. Upper semicontinuity,
	\item P6. Consistency, and
	\item P7. Convexity of the contours.
\end{itemize}
It is clear that all of them are desirable properties for depth function of compact convex sets. However, not all of them have to be part of an axiomatic definition. For instance, it seems appropriate to have either P3a. and P4a. or P3b. and P4b. At the same time, P7., although important,  does not belong to any of the existing axiomatic definitions and  P5. and a general case of P6. only belong to the functional (metric) axiomatic definition of statistical depth.

Taking all of this into account, we propose to consider:
\begin{itemize}
	\item The \emph{algebraic} depth of compact convex sets, when properties P1., P2., P3a. and P4a. are satisfied.
	\item The  \emph{restricted algebraic} depth  of compact convex sets, when properties P1., P2., P3a., P4a. P5., P6. and P7. are satisfied.
	\item The \emph{geometric} depth  of compact convex sets, when properties P1., P2., P3b. and P4b. are satisfied.
	\item The  \emph{restricted geometric} depth  of compact convex sets, when properties P1., P2., P3b., P4b. P5., P6. and P7. are satisfied.
\end{itemize}
Note that the algebraic depth can be considered as an adaptation of the notions of multivariate depth and of semilinear fuzzy depth. Meanwhile, the geometric depth can be seen as a conversion of the geometric fuzzy depth  and the restricted geometric depth as a modification of the functional (metric) depth.

We have studied the satisfaction of the above properties for the Tukey depth of compact convex sets, which is an adaptation to this setting of the multivariate Tukey depth and a simplification of the Tukey for fuzzy sets. It happens that this depth function satisfies all of these properties but for P3b., for which we have provided a counterexample. Thus, the Tukey depth of compact convex sets is a restricted algebraic depth and, in particular, an algebraic depth. However, it is not a geometric depth and, consequently, neither a restricted geometric depth.

\citet{Cascos} proposed a notion of depth for random closed sets. They require properties P1, P5 (for the Fell topology instead of the Hausdorff metric) and the property that a degenerate random set should assign depth 1 to its only value and 0 to any other random set. Admitting unbounded sets as values leads to some defining properties of depth being hard to adapt, a situation they solve by opting for a minimal list of properties. It is worth mentioning that, in the case of compact convex values, convergence in the Fell topology and in the Hausdorff metric are equivalent \cite[Corollary 3A]{Salinetti}. Hence both upper semicontinuity requirements are equivalent for the Tukey depth and Proposition \ref{muj} provides a proof of upper semicontinuity with respect to the Fell topology. Such a proof is missing in \cite{Cascos} on the grounds of it being `easy' (a direct proof without invoking extra facts does not seem to be that easy).

{\bf Acknowledgments}
A. Nieto-Reyes and L.Gonzalez are supported by the Spanish Ministerio de Ciencia, Innovaci\'on y Universidades grant MTM2017-86061-C2-2-P. P. Ter\'an is supported by the Ministerio de Econom\'\i a y Competitividad grant MTM2015-63971-P, the
Ministerio de Ciencia, Innovación y Universidades grant PID2019-104486GB-I00 and the
Consejer\'\i a de Empleo, Industria y Turismo del Principado de Asturias grant GRUPIN-IDI2018-000132.


\begin{thebibliography}{999}
		
		\bibitem[Gil et al.(2002)]{GilBlood}
		Gil, M.\'A., Lubiano, M.A., Montenegro, M. and L\'opez, M.T. Least squares fitting of an affine function and strength of association for interval-valued data. {\em Metrika} {\bf 2002}, {\em 56(2)}, 97--111.
		
		\bibitem[Lima Neto and de Carvalho(2017)]{soccer}
		Lima Neta, E. de A. and de Carvalho, F. de A.T. Nonlinear regression applied to interval-valued data. {\em Patt. Anal. App.} {\bf 2017}, {\em 20}, 809--824.
		
		
		\bibitem[Molchanov(2017)]{molchanov}
		Molchanov, I. \textit{Theory of Random Sets}, 2nd ed.;Springer, London, 2017.
		
		\bibitem[Arstein and Vitale(1975)]{arstein}
		Artstein, Z. and Vitale, R.A. A strong law of large numbers for random compact convex sets. {\em Ann. Probab.} {\bf 1975}, 879--828.
		
		\bibitem[Gonz\'alez-Rodr\'iguez et al.(2007)]{gonzalezcolubi}
		Gonz\'alez-Rodr\'iguez, G., Blanco, A., Corral, N. and Colubi, A. Least squares estimation of linear regression models for convex compact convex random sets. {\em Adv. Data Anal. Classif.} {\bf 2007}, {\em 1(1)}, 67--81.
		
		\bibitem[Sinova et al.(2010)]{sinovamedian}
		Sinova, B., Casals, M.R., Colubi, A. and Gil, M.\'A. The median of a random interval. In: Combining Soft Computing and Statistical Methods in Data Analysis, Springer, Berlin, Heidelberg, 2010; pp. 575--583.
		
		\bibitem[Richey and Sarkar (2022)]{intersect}
		Richey, J.; Sarkar, A. Intersections of random sets. {\em J. Appl. Probab.} {\bf 2022}, {\em 59(1)}, 131--151.
		
		\bibitem[Shi et al.(2022)]{Pen}
		Shi, P.; Lu, L.; Fan, X.; Xin, Y.; Ni, J.  A novel underwater sonar image enhancement algorithm based on approximation spaces of random sets. {\em Multimed. Tools. Appl.}  {\bf 2022}, {\em 81(4)},  4569--4584.
		
		\bibitem[Jörnsten(2004)]{clustering}
		Jörnsten, R. Clustering and classification based on the L1 data depth. {\em J. Multivariate Anal.} {\bf 2004}, {\em 90(1)}, 67--89.
		
		\bibitem[Nieto-Reyes et al.(2021)]{aplicacion1}
		Nieto-Reyes,A.;Battey, H.; Francisci, G. Functional Symmetry and Statistical Depth for the Analysis of Movement Patterns in Alzheimer's Patients. {\em Mathematics} {\bf 2021}, {\em 9(8)}, 820.
		
		\bibitem[Nieto-Reyes et al.(2021)]{aplicacion2}
		Nieto-Reyes, A.;Duque, R.;Francisci, G. A Method to Automate the Prediction of Student Academic Performance from Early Stages of the Course. {\em Mathematics} {\bf 2021}, {\em 9(21)}, 2677.
		
		\bibitem[Liu(1990)]{LiuSimplicial}
		Liu, R.Y. On a notion of data depth based on random simplices. {\em Ann. Statist.} {\bf 1990}, {\em 18}, 405--414.
		
		\bibitem[Zuo and Serfling(2000)]{ZuoSerfling}
		Zuo, Y.; Serfling, R. General notions of statistical depth function. {\em Ann. Statist.} {\bf 2000}, {\em 28}, 461--482.
		
		\bibitem[Nieto-Reyes and Battey(2021)]{NietoBatteyJMVA}
		Nieto-Reyes, A.;  Battey, H. A topologically valid construction of depth for functional data. {\em Journal of Multivariate Analysis} {\bf 2021}, {\em 184}, 104738. 
		
		\bibitem[Gonz\'alez-De La Fuente et al.(202Xa)]{primerarticulo}
		G\'onzalez-de la Fuente, L.; Nieto-Reyes, A.; Ter\'an, P. Statistical depth for fuzzy sets. {\em Fuzzy Sets and Systems}, to appear. https://doi.org/10.1016/j.fss.2021.09.015
		
		\bibitem[Nieto-Reyes and Battey(2016)]{NietoBattey}
		Nieto-Reyes, A.; Battey, H. A topologically valid definition of depth for functional data. {\em Statist. Sci.} {\bf 2016}, {\em 31}, 61--79.
		
		\bibitem[Gonz\'alez-De La Fuente et al.(202Xb)]{CaptFN}
		G\'onzalez-de la Fuente, L.; Nieto-Reyes, A.; Ter\'an, P. Two notions of depth in the fuzzy setting.  In L. García-Escudero, A. Gordaliza, A. Mayo, M.A. Lubiano Gomez, M.A. Gil, P. Grzegorzewski, O. Hryniewicz (Eds.), Building Bridges between Soft and Statistical Methodologies for Data Science, Springer, Chapter 30, to appear.
		
		\bibitem[Tukey(1975)]{tukey}
		Tukey, J.W. Mathematics and Picturing Data. In: Proceedings of the International Congress of Mathematicians, Vancouver, BC, Canada, 21–29 August 1974; Canadian Mathematical Congress: Montreal, QC, Canada, 1975; pp. 523–-531.
		
		\bibitem[Serfling(2002)]{SerflingSpatial}
		Serfling, R. A depth function and a scale curve based on spatial quantiles. In Y.Dodge (Ed.) Statistical Data Analysis Based on $L_{1}$-norm and Related Methods , Birkhäuser, Basel, 2002, 25--38.
		
		\bibitem[Cuesta-Albertos and Nieto-Reyes(2008)]{CuestaNieto}
		Cuesta-Albertos, J.A.; Nieto-Reyes, A. The random Tukey depth. {\em Comput. Statist. Data Anal.} {\bf 2008}, {\em 52}, 4979--4988.
		
		\bibitem[Chakraborty and Chaudhuri(2014)]{SpatIndInf}
		Chakraborty, A.; Chaudhuri, P. The spatial distribution in infinite dimensional spaces and related quantiles and depths. {\em The Annals of Statistics} {\bf  2014}, {\em 42(3)}, 1203--1231.
		
		\bibitem[Cuesta-Albertos and Nieto-Reyes(2010)]{CuestaNietoMieres}
		Cuesta-Albertos, J.A.; Nieto-Reyes, A.
		Functional classification and the random Tukey depth. Practical issues.
		In C. Borgelt, G. González-Rodríguez, W. Trutsching, M.A. Lubiano, M.A. Gil, P. Grzegorzewski, O. Hryniewicz (Eds.), Combining Soft Computing and Statistical Methods in Data Analysis, vol. 77, Springer, Berlin 2010, pp. 123-130
		
		\bibitem[Gonz\'alez-De La Fuente et al.(202Xc)]{CaptFT}
		G\'onzalez-de la Fuente, L.; Nieto-Reyes, A.; Ter\'an, P. Tukey depth for fuzzy sets.  In L. García-Escudero, A. Gordaliza, A. Mayo, M.A. Lubiano Gomez, M.A. Gil, P. Grzegorzewski, O. Hryniewicz (Eds.), Building Bridges between Soft and Statistical Methodologies for Data Science, Springer, to appear.
		
		\bibitem[Cascos et al.(2021)]{Cascos}
		Cascos, I.; Li, Q.; Molchanov, I. Depth and outliers for samples of sets and random sets distributions. {\em Aust. N. Z. Stat.} {\bf 2021}, {\em 63}, 55--82.
		
		\bibitem[Matheron(1975)]{randomsetsdefinition}
		Matheron, G. \textit{Random sets and integral geometry}; Wiley, New York, 1975.
		
		\bibitem[Himmelberg(1974)]{randomsetstheorem}
		Himmelberg, C. Measurable relations. {\em Fund. Math.} {\bf 1974}, {\em 87}, 53--72.
		
		\bibitem[Bonnensen and Fenchel(1948)]{Bonnensen}
		Bonnensen, T.; Fenchel, W. \textit{Theorie der Konvexen Korper}; Chelsea, New York, 1948.
		
		\bibitem[Zadeh(1975)]{extesion}
		Zadeh, L.A. The concept of a linguistic variable and its application to approximate reasoning, Part 1, {\em Inform. Sci.} {\bf 1975}, {\em 8}, 199--249, Part 2, {\em Inform. Sci.} {\bf 1975}, {\em 8}, 301--353, Part 3, {\em Inform. Sci.} {\bf 1975}, {\em 8}, 43--80.
		
		\bibitem[Gruber and Lettl(1980)]{Gruber}
		Gruber, P. M.; Lettl, G. Isometries of the Space of Convex Bodies in Euclidean Space. Bulletin of the London Mathematical Society 12, 455--462.
		
		\bibitem[Vitale(1985)]{Vitale}
		Vitale, R. A. $L_p$ metrics for compact, convex sets. J. Approx. Theory 45, 280--287.
		
		
		
		\bibitem[Massart(1990)]{Massart}
		Massart, P. The tight constant in the Dvoretzky--Kiefer--Wolfowitz inequality. Ann. Probab. 18, 1269--1283.
		
		\bibitem[Gin\'e and Nickl(2016)]{Nickl}
		Gin\'e, E.; Nickl, R. Mathematical foundations of infinite-dimensional statistical models. Cambridge University Press, Cambridge, 2016.
		
		
		\bibitem[Donoho and Gasko(1992)]{Donoho}
		Donoho, D.L.; Gasko, M. Breakdown properties of location estimates based on halfspace depth and projected outlyinges. {\em Annals of Statistics} {\bf 1992}, {\em 20(4)}, 1803--1827.
		
		\bibitem[Serfling(2003)]{Ser}
		Serfling, R. Depth Functions in Nonparametric Multivariate Inference. In: Data Depth: Robust Multivariate     Analysis, Computational Geometry and Applications, DIMACS Series in Discrete Mathematics and Theoretical Computer Science, 2003.
		
		
		\bibitem[Salinetti and Wets(1979)]{Salinetti} Salinetti, G.; Wets, R. J. B. On the convergence of sequences of convex sets in finite dimensions. SIAM Review 21, 18--33.
		
		
		
	\end{thebibliography}
\end{document}